\title{Long paths and Hamiltonicity in random graphs}
\author{Michael Krivelevich
\thanks{School of Mathematical Sciences, Raymond and Beverly
Sackler Faculty of Exact Sciences, Tel Aviv University, Tel Aviv,
6997801, Israel. Email: krivelev@post.tau.ac.il. Research supported in
part by a USA-Israel BSF grant and by a grant from the Israel
Science Foundation.}}
\date{}
\documentclass[11pt]{article}
\usepackage{amsmath,amssymb,latexsym}

\usepackage{tikz}

\newtheorem{theo}{Theorem}[section]
\newtheorem{prop}[theo]{Proposition}
\newtheorem{lemma}[theo]{Lemma}
\newtheorem{coro}[theo]{Corollary}

\newtheorem{defin}[theo]{Definition}

\newenvironment{proof}{\noindent{\bf Proof\,}}{\hfill$\Box$}

\begin{document}
\maketitle
\begin{abstract}
We discuss several classical results about long paths and Hamilton cycles in random graphs and present accessible versions of their proofs, relying on the Depth First Search (DFS) algorithm and the notion of boosters.
\end{abstract}

%
\section{Introduction}
Long paths and Hamiltonicity are certainly between  the most central and researched topics of modern graph theory. It is thus only natural to expect that they will take a place of honor in the theory of random graphs. And indeed, the typical appearance of long paths and of Hamilton cycles is one of the most thoroughly studied directions in random graphs, with great many diverse and beautiful results obtained over the last fifty or so years.

In this survey we aim to cover some of the most basic theorems about long paths and Hamilton cycles in the classical models of random graphs, such as the binomial random graph or the random graph process. At no means this text should be viewed as a comprehensive coverage of results of this type in various models of random graphs; the reader looking for breadth should rather consult research papers, or a forthcoming monograph on random graphs due to Frieze and Karo\'nski \cite{FK-book}. Instead, we focus on simplicity, aiming to provide accessible proofs of several classical results on the subject, and showcasing the tools successfully applied recently to derive new and fairly simple proofs, such as applications of the Depth First Search (DFS) Algorithm for finding long paths in random graphs, and the notion of boosters.

Although this text should be fairly self-contained mathematically, basic familiarity and hands-on experience with random graphs would certainly be of help for the prospective reader. The standard random graph theory monographs of Bollob\'as \cite{Bol-book} and of Janson, \L uczak and Ruci\'nski \cite{JLR} certainly provide (much more than) the desired background.

The text is based on a mini-course with the same name, delivered by the author at the LMS-EPSRC Summer School on Random Graphs, Geometry and Asymptotic Structure, organized by Dan Hefetz and Nikolaos Fountoulakis at the University of Birmingham in the summer of 2013. The author would like to thank the course organizers for inviting him to deliver the mini-course, and for encouraging him to create lecture notes for the course, that eventually served as a basis for the present text.

\section{Tools}\label{s-tools}
In this section we gather notions and tools to be applied later in the proofs. They include standard graph theoretic notation, asymptotic estimates of binomial coefficients, Chebyshev's and Chernoff's inequalities, basic notation from random graphs (Section \ref{ss-prelim}), as well as algorithmic and combinatorial tools --- the Depth First Search (DFS) algorithm for graph exploration (Section \ref{ss-DFS}), the so called rotation-extension technique of P\'osa, and the notion of boosters (Section \ref{ss-boosters}).

\subsection{Preliminaries}\label{ss-prelim}
\subsubsection{Notation and terminology}
Our graph theoretic notation and terminology are fairly standard. In particular, for a graph $G=(V,E)$ and disjoint vertex subsets $U,W\subset V$, we denote by $N_G(U)$ the external neighborhood of $U$ in $G$: $N_G(U)=\{v\in V-U: v\mbox{ has a neighbor in $U$}\}$. The number of eddes of $G$ spanned by $U$ is denoted by $e_G(U)$, the number of edges of $G$ between $U$ and $W$ is $e_G(U,W)$. When the graph $G$ is clear from the context, we may omit $G$ in the subscript in the above notation.

Path and cycle lengths are measured in edges.

When dealing with graphs on $n$ vertices, we will customarily use $N$ to denote the number of pairs of vertices in such graphs: $N=\binom{n}{2}$.

\subsubsection{Asymptotic estimates}\label{sss-asymp}
We will use the following, quite standard and easily proven, estimates of binomial coefficients. Let $1\le x\le k\le n$ integers. Then
\begin{eqnarray}
\left(\frac{n}{k}\right)^k\le \binom{n}{k}&\le& \left(\frac{en}{k}\right)^k\,,\label{bin1}\\
\frac{\binom{n-x}{k-x}}{\binom{n}{k}}&\le& \left(\frac{k}{n}\right)^x\,,\label{bin2}\\
\frac{\binom{n-x}{k}}{\binom{n}{k}}&\le& e^{-\frac{kx}{n}}\,.\label{bin3}
\end{eqnarray}

\subsubsection{Chebyshev and Chernoff}
Chebyshev's Inequality helps to show concentration of a random variable $X$ around its expectation, based of the first two moments of $X$. It reads as follows: let $X$ be a random variable with expectation $\mu$ and variance $\sigma^2$. Then for any $a>0$,
$$
Pr[|X-\mu|\ge a\sigma]\le \frac{1}{a^2}\,.
$$

The following are very standard bounds on the lower and
the upper tails of the Binomial distribution due to Chernoff:
If $X \sim Bin(n,p)$, then
\begin{itemize}
    \item $\Pr\left(X<(1-a)np\right)<\exp\left(-\frac{a^2np}{2}\right)$ for every $a>0.$
    \item $\Pr\left(X>(1+a)np\right)<\exp\left(-\frac{a^2np}{3}\right)$ for every $0 < a < 1.$
\end{itemize}
Another, trivial yet useful, bound is as follows:
Let $X \sim Bin(n,p)$ and $k \in \mathbb{N}$. Then $$\Pr(X \geq k)
\leq \left(\frac{enp}{k}\right)^k.$$  Indeed,
$\Pr(X \geq k) \leq \binom{n}{k}p^k \leq
\left(\frac{enp}{k}\right)^k$.

\subsubsection{Random graphs, asymptotic notation}
As usually, $G(n,p)$ denotes the probability space of graphs with vertex set $\{1,\ldots,n\}=[n]$, where every pair of distinct elements of $[n]$ is an edge of $G\sim G(n,p)$ with probability $p$, independently of other pairs. For $0\le m\le N$, $G(n,m)$ denotes the probability space of all graphs with vertex set $[n]$ and exactly $m$ edges, where all such graphs are equiprobable: $Pr[G]=\binom{N}{m}^{-1}$. One can expect that the probability spaces $G(n,p)$ and $G(n,m)$ have many similar features, when the corresponding parameters are appropriately tuned: $m=Np$; accurate quantitative statements are available, see \cite{Bol-book}, \cite{JLR}. This similarity frequently allows to prove a desired property for one of the probability spaces, and then to transfer it to the other one.

We will also address briefly the model $D(n,p)$ of directed random graphs, defined as follows: the vertex set is $[n]$, and each of the $n(n-1)=2N$ ordered pairs $1\le i\ne j\le n$ is a directed edge of $D(n,p)$ with probability $p$, independently from other pairs.

We say that an event ${\cal E}_n$ occurs with high probability, or {\bf whp} for brevity, in the probability space $G(n,p)$ if $\lim_{n\rightarrow\infty} Pr[G\sim G(n,p)\in {\cal E}_n]=1$. (Formally, one should rather talk about a sequence of events $\{{\cal E}_n\}_n$ and a sequence of probability spaces $\{G(n,p)\}_n$.) This notion is defined in other (sequences of) probability spaces in a similar way.

Let $k\ge 2$ be an integer, and assume that $0\le p,p_1,\ldots,p_k\le 1$ satisfy $(1-p)=\prod_{i=1}^k (1-p_i)$. Then the random graphs $G\sim G(n,p)$ and $G'=\bigcup_{i=1}^k G(n,p_i)$ have the exact same distribution. Indeed, it is obvious that each pair of vertices $1\le i<j\le n$ is an edge in both graphs $G,G'$ independently of other pairs. In $G$, this edge does not appear with probability $1-p$, and in order for it not appear in $G'$, it should not appear in any of the random graphs $G(n,p_i)$ -- which happens with probability $\prod_{i=1}^k(1-p_i)=1-p$, the same probability as in $G$. This very useful trick is called {\em multiple exposure} as it allows to generate (to expose) a random graph $G\sim G(n,p)$ in stages, by generating the graphs $G(n,p_i)$ sequentially and then by taking their union. In case when the last probability $p_k$ is much smaller than the rest, it is also called {\em sprinkling} -- a typical scenario in this case is to expose the bulk of the random graph $G\sim G(n,p)$ first by generating the graphs $G(n,p_i)$, $i=1,\ldots,k-1$, to come close to a target graph property $P$, and then to add few random edges from the last random graph $G(n,p_k)$ (to sprinkle these few edges) to finish off the job.

\subsection{Depth First Search and its applications for finding long paths}\label{ss-DFS}

The {\em Depth First Search} is a well known graph exploration algorithm, usually applied to discover connected components of an input graph. As it turns out, this algorithm is particularly suitable for finding long paths in graphs, and using it in the context of random graphs can really make wonders. We will see some of them later in this text.

Recall that the DFS (Depth First Search) is a graph search algorithm
that visits all vertices of a (directed or undirected) graph. The algorithm receives as an input a graph $G=(V,E)$; it is also assumed that an order $\pi$ on the
vertices of $G$ is given, and the algorithm prioritizes vertices
according to $\pi$. The algorithm maintains three sets of vertices, letting
$S$ be the set of vertices whose exploration is complete, $T$ be the
set of unvisited vertices, and $U=V\setminus(S\cup T)$, where the
vertices of $U$ are kept in a stack (the last in, first out data
structure). It initiliazes with $S=U=\emptyset$ and
$T=V$, and runs till $U\cup T=\emptyset$. At each round of the
algorithm, if the set $U$ is non-empty, the algorithm queries $T$
for neighbors of the last vertex $v$ that has been added to $U$,
scanning $T$ according to $\pi$. If $v$ has a neighbor $u$  in
$T$, the algorithm deletes $u$ from $T$ and inserts it into $U$. If
$v$ does not have a neighbor in $T$, then $v$ is popped out of $U$
and is moved to $S$. If $U$ is empty, the algorithm chooses the
first vertex of $T$ according to $\pi$, deletes it from $T$ and
pushes it into $U$. In order to complete the exploration of the
graph, whenever the sets $U$ and $T$ have both become empty (at this
stage the connected component structure of $G$ has already been
revealed), we make the algorithm query all remaining pairs of
vertices in $S=V$, not queried before. Figure \ref{fig1} provides an illustration of applying the DFS algorithm.
\begin{figure}[h]
  \centering
  \begin{tabular}{ccc}
    \begin{tikzpicture}[scale=1.0,vertex/.style={draw,circle,color=black,fill=black,inner sep=1}]
      \pgfmathsetmacro{\scal}{2.0}

      \pgfmathsetmacro{\vert}{\scal * 0.866} 
      \pgfmathsetmacro{\verp}{\scal * 1.0} 
      \pgfmathsetmacro{\hori}{\scal * 0.5} 
      \pgfmathsetmacro{\spac}{\scal * 2.1} 

      \clip (-\hori - 0.25, 0.65) rectangle (\spac+\hori+0.25,-\vert-2*\verp-2.5);

      \node[vertex,label=$1$] (1) at (0,0) {};
      \node[vertex,label=$2$] (2) at (\spac, 0) {};
      \node[vertex,label=below:$3$] (3) at (-\hori, -\vert) {};
      \node[vertex,label=below right:$4$] (4) at (\spac-\hori, -\vert) {};
      \node[vertex,label=right:$5$] (5) at (\spac-\hori, -\vert-2*\verp) {};
      \node[vertex,label=right:$6$] (6) at (\spac-\hori, -\vert-\verp){};
      \node[vertex,label=below:$7$] (7) at (\spac+\hori, -\vert) {};
      \node[vertex,label=below:$8$] (8) at (\hori, -\vert) {};

      \draw (1) -- (3) -- (8) -- (1);
      \draw (2) -- (4) -- (7) -- (2);
      \draw (4) -- (6) -- (5);
      \draw[bend right=45] (2) to(6);
    \end{tikzpicture}

    &

    \hspace{1em}

    &
    \small{
    \begin{tabular}[b]{| c | c | c | c |}
      \hline
      Step& $S$             & $U$        &$T$ \\ \hline
      0   &$\emptyset$      &$\emptyset$ &$\{1,\ldots, 8\}$   \\ \hline
      1   &$\emptyset$      &$1$         &$\{2,\ldots, 8\}$   \\ \hline
      2   &$\emptyset$      &$1,3$       & $\{2,4,\ldots,8\}$ \\ \hline
      3   &$\emptyset$      &$1,3,8$     & $\{2,4,\ldots,7\}$ \\ \hline
      4   &$\{8\}$          &$1,3$       & $\{2,4,\ldots,7\}$ \\ \hline
      5   &$\{3,8\}$        &$1$         & $\{2,4,\ldots,7\}$ \\ \hline
      6   &$\{1,3,8\}$      &$\emptyset$ & $\{2,4,\ldots,7\}$ \\ \hline
      7   &$\{1,3,8\}$      &$2$         & $\{4,5,6,7\}$      \\ \hline
      8   &$\{1,3,8\}$      &$2,4$       & $\{5,6,7\}$        \\ \hline
      9   &$\{1,3,8\}$      &$2,4,6$     & $\{5,7\}$          \\ \hline
      10  &$\{1,3,8\}$      &$2,4,6,5$   & $\{7\}$            \\ \hline
      11  &$\{1,3,5,8\}$    &$2,4,6$     & $\{7\}$            \\ \hline
      12  &$\{1,3,5,6,8\}$  &$2,4$       & $\{7\}$            \\ \hline
      13  &$\{1,3,5,6,8\}$  &$2,4,7$     & $\emptyset$        \\ \hline
      14  &$\{1,3,5,6,7,8\}$&$2,4$       & $\emptyset$        \\ \hline
      15  &$\{1,3,4,5,6,7,8\}$&$2$       & $\emptyset$        \\ \hline
      16  &$\{1,\ldots,8\}$  &$\emptyset$ & $\emptyset$        \\ \hline
    \end{tabular}
  }
  \end{tabular}
\caption{Graph $G$ with vertices labeled is on the left, and the protocol of applying the DFS algorithm to $G$ is on the right. Observe that at any point of the algorithm execution the set $U$ spans a path in $G$. }\label{fig1}
\end{figure}

Observe that the DFS algorithm starts revealing a connected
component $C$ of $G$ at the moment the first vertex of $C$ gets into
(empty beforehand) $U$ and completes discovering all of $C$ when $U$
becomes empty again. We call a period of time between two
consecutive emptyings of $U$ an {\em epoch}, each epoch corresponding
to one connected component of $G$. During the execution of the DFS algorithm as depicted in Figure \ref{fig1}, there are two components in the graph $G$, and respectively there are two epochs -- the first is Steps 1--6, and the second is Steps 7--16.

The following properties of the DFS algorithm are immediate to verify:
\begin{itemize}
\item[{\bf (D1)}] at each round of the algorithm one vertex moves, either from
$T$ to $U$, or from $U$ to $S$;
\item[{\bf (D2)}] at any stage of the algorithm, it has been revealed already
that the graph $G$ has no edges between the current set $S$ and the
current set $T$;
\item[{\bf (D3)}] the set $U$ always spans a path (indeed, when a vertex $u$ is
added to $U$, it happens because $u$ is a neighbor of the last
vertex $v$ in $U$; thus, $u$ augments the path spanned by $U$, of
which $v$ is the last vertex).
\end{itemize}

We now exploit the features of the DFS algorithm to derive the existence of long paths in expanding graphs.

\begin{prop}\label{DFS1}
Let $k,l$ be positive integers. Assume that $G=(V,E)$ is a graph on more than $k$ vertices, in which every vertex subset $S$ of size $|S|=k$ satisfies: $|N_G(S)|\ge l$. Then $G$ contains a path of length $l$.
\end{prop}

\begin{proof}
Run the DFS algorithm on $G$, with $\pi$ being an arbitrary ordering of $V$. Look at the moment during the algorithm execution when the size of the set $S$ of already processed vertices becomes exactly equal to $k$ (there is such a moment as the vertices of $G$ move into $S$ one by one, till eventually all of them land there). By Property {\bf (D2)} above, the current set $S$ has no neighbors in the current set $T$, and thus $N(S)\subseteq U$, implying $|U|\ge l$.  The last move of the algorithm was to shift a vertex from $U$ to $S$, so before this move the set $U$ was one vertex larger. The set $U$ always spans a path in $G$, by Property {\bf (D3)}. Hence $G$ contains a path of length $l$.
\end{proof}

\medskip

\begin{prop}\label{DFS2}\cite{BKS}
Let $k<n$ be positive integers. Assume that $G=(V,E)$ is a graph on $n$ vertices, containing an edge between any two disjoint subsets $S,T\subset V$ of size $|S|=|T|=k$. Then $G$ contains a path of length $n-2k+1$ and a cycle of length at least $n-4k+4$.
\end{prop}

\begin{proof}
Run the DFS algorithm on $G$, with $\pi$ being an arbitrary ordering of $V$. Consider the moment during the algorithm execution when $|S|=|T|$ --- there is such a moment, by Property {\bf (D1)}. Since $G$ has no edges between the current $S$ and the current $T$ by Property {\bf (D2)}, it follows by the proposition's assumption that both sets $S$ and $T$ are of size at most $k-1$. This leaves us with the set $U$ whose size satisfies: $|U|\ge n-2k+2$. Since $U$ always spans a path by {\bf (D3)}, we obtain a path $P$ of desired length. To argue about a cycle, take the first and the last $k$ vertices of $P$. By the proposition's assumption there is an edge between these two sets, this edge obviously closes a cycle with $P$, whose length is at least $n-4k+4$, as required.
\end{proof}

\medskip

As we have hinted already, the DFS algorithm is well suited to handle directed graphs too. Similar results to those stated above can be obtained for the directed case. Here is an analog of Proposition \ref{DFS2} for directed graphs; the proof is the same, mutatis mutandis.

\begin{prop}\label{DFS3}\cite{BKS}
Let $k<n$ be positive integers. Let $G=(V,E)$ be a directed graph on $n$ vertices, such that for any ordered pair of disjoint subsets $S,T\subset V$ of size $|S|=|T|=k$, $G$ has a directed edge from $S$ to $T$. Then $G$ contains a directed path of length $n-2k+1$ and a directed cycle of length at least $n-4k+4$.
\end{prop}

\subsection{P\'osa's Lemma and boosters}\label{ss-boosters}
In this section we present yet another technique for showing the existence of long paths in graphs. This technique, introduced by P\'osa in 1976 \cite{Pos76} in his research on Hamiltonicity of random graphs, is applicable not only for arguing about long paths, but also for various Hamiltonicity questions. And indeed, we will see its application in this context later.

In quite informal terms, P\'osa's Lemma guarantees that expanding graphs not only have long paths, but also provide a very convenient backbone for augmenting a graph to a Hamiltonian one by adding new (random) edges. The fact that expanders are good for getting long paths is already not new to us --- Propositions \ref{DFS1} and \ref{DFS2} are just about this. P\'osa's Lemma however quantifies things somewhat differently, and as a result yields further benefits.

We start by defining formally the notion of an expander.

\begin{defin}\label{expander}
For a positive integer $k$ and a positive real $\alpha$, a graph $G=(V,E)$ is a $(k,\alpha)$-expander if
$|N_G(U)|\ge \alpha |U|$ for every
subset $U\subset V$ of at most $k$ vertices.
\end{defin}

By the way of example, Proposition \ref{DFS1} can now be rephrased (in a somewhat weaker form -- we now require the expansion of all sets of size {\em up to} $k$) as follows: if $G$ is a $(k,\alpha)$-expander, then $G$ has a path of length at least $\alpha k$. For technical reasons, P\'osa's Lemma uses the particular case $\alpha=2$.

The idea behind P\'osa's approach is fairly simple and natural --- one can start with a (long) path $P$, and then, using extra edges, perform a sequence of simple deformations (rotations), till it will be possible to close the deformed path to a cycle, or to extend it by appending a vertex outside $V(P)$; then this can be repeated if necessary to create an even longer path, or to close it to a Hamilton cycle. The approach is thus called naturally the {\em rotation-extension} technique. Formally, let $P=x_0x_1\ldots x_h$ be a path in a graph $G=(V,E)$, starting at a vertex $x_0$. Suppose $G$ contains an edge $(x_i,x_h)$ for some
$0\le i<h-1$. Then a new path $P'$ can be obtained by rotating the
path $P$ at $x_i$, i.e. by adding the edge $(x_i,x_h)$ and erasing
$(x_i,x_{i+1})$. This operation is called an {\em elementary
rotation} and it depicted in Figure \ref{fig2}. Note that the obtained path $P'$ has the same length $h$ and
starts at $x_0$. We can therefore apply an elementary rotation to
the newly obtained path $P'$, resulting in a path $P''$ of length
$h$, and so on. If after a number of rotations an endpoint $x$ of
the obtained path $Q$ is connected by an edge to a vertex $y$
outside $Q$, then $Q$ can be extended by adding the edge $(x,y)$.

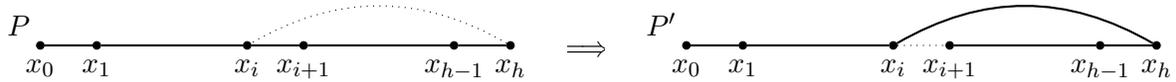
\begin{figure}
  \begin{tabular}{ccc}
    \begin{tikzpicture}[vertex/.style={draw,circle,color=black,fill=black,inner sep=1}]
      \pgfmathsetmacro{\elen}{0.75}
      \pgfmathsetmacro{\plen}{2.0}

      \node[above left] at (0,0) {$P$};

      \node[vertex,label=below:$x_0$] (0) at (0, 0) {};
      \node[vertex,label=below:$x_1$] (1) at (\elen, 0) {};
      \node[vertex,label=below:$x_i$] (i) at (\elen+\plen,0) {};
      \node[vertex,label=below:$x_{i+1}$] (ip) at (2*\elen+\plen,0) {};
      \node[vertex,label=below:$x_{h-1}$] (nm) at (2*\elen+2*\plen,0) {};
      \node[vertex,label=below:$x_h$] (n) at (3*\elen+2*\plen,0) {};

      \clip (-0.5, -0.5) rectangle (3*\elen+2*\plen+0.25,0.65);

      \draw[thick] (0) -- (1) -- (i) -- (ip) -- (nm) -- (n);
      \draw[dotted,bend right=30] (n) to (i);
    \end{tikzpicture}

    &

    \begin{tikzpicture}
      \clip (-0.25, -0.5) rectangle (0.25,0.65);

      \node (Posa) at (0,0) {$\Longrightarrow$};
    \end{tikzpicture}

    &

    \begin{tikzpicture}[vertex/.style={draw,circle,color=black,fill=black,inner sep=1}]
      \pgfmathsetmacro{\elen}{0.75}
      \pgfmathsetmacro{\plen}{2.0}

      \node[above left] at (0,0) {$P'$};

      \node[vertex,label=below:$x_0$] (0) at (0, 0) {};
      \node[vertex,label=below:$x_1$] (1) at (\elen, 0) {};
      \node[vertex,label=below:$x_i$] (i) at (\elen+\plen,0) {};
      \node[vertex,label=below:$x_{i+1}$] (ip) at (2*\elen+\plen,0) {};
      \node[vertex,label=below:$x_{h-1}$] (nm) at (2*\elen+2*\plen,0) {};
      \node[vertex,label=below:$x_h$] (n) at (3*\elen+2*\plen,0) {};

      \clip (-0.5, -0.5) rectangle (3*\elen+2*\plen+0.25,0.65);

      \draw[thick] (0) -- (1) -- (i);
      \draw[dotted] (i) -- (ip);
      \draw[thick] (ip) -- (nm) -- (n);
      \draw[bend right=30,thick] (n) to (i);
    \end{tikzpicture}
  \end{tabular}
  \caption{Elementary rotation is applied to a path $P$ to get a new path $P'$ with the same vertex set.}\label{fig2}
\end{figure}


The power of the rotation-extension technique of P\'osa hinges on
the following lemma.

\begin{lemma}\label{lem-Posa}\cite{Pos76}
Let $G$ be a graph, $P$ be a longest path in $G$ and ${\cal P}$ be
the set of all paths obtainable from $P$ by a sequence of elementary
rotations. Denote by $R$ the set of ends of paths in ${\cal P}$, and
by $R^-$ and $R^+$ the sets of vertices immediately preceding and
following the vertices of $R$ along $P$, respectively. Then $ N_G(R)
\subset R^-\cup R^+ $.
\end{lemma}

\begin{proof} Let $x\in R$ and $y\in V(G)\setminus (R\cup R^-\cup R^+)$,
and consider a path $Q\in{\cal P}$ ending at $x$. If $y\in
V(G)\setminus V(P)$, then $(x,y)\not\in E(G)$, as otherwise the path
$Q$ can be extended by adding $y$, thus contradicting our assumption
that $P$ is a longest path. Suppose now that $y\in V(P)\setminus
(R\cup R^-\cup R^+)$. Then $y$ has the same neighbors in every path
in ${\cal P}$, because an elementary rotation that removed one of
its neighbors along $P$ would, at the same time, put either this
neighbor or $y$ itself in $R$ (in the former case $y\in R^-\cup
R^+$).  Then if $x$ and $y$ are adjacent, an elementary rotation
applied to $Q$ produces a path in ${\cal P}$ whose endpoint is a
neighbor of $y$ along $P$, while belonging itself to $R$, a contradiction. Therefore in both cases
$x$ and $y$ are non-adjacent.
\end{proof}

\medskip

The following immediate consequence of Lemma \ref{lem-Posa} is frequently applied in Hamiltonicity problems in random graphs.
\begin{coro}\label{cor-Posa2}
Let $h,k$ be positive integers. Let $G=(V,E)$ be a graph such that
its longest path has length $h$, but it contains no cycle of length
$h+1$. Suppose furthermore that $G$ is a $(k,2)$-expander. Then there
are at least $\frac{(k+1)^2}{2}$ non-edges in $G$ such that if any
of them is turned into an edge, then the new graph contains an
$(h+1)$-cycle.
\end{coro}
\begin{proof} Let $P=x_0x_1\ldots x_h$ be a longest path in $G$ and let
$R,R^-,R^+$ be as in Lemma \ref{lem-Posa}. Notice that $|R^-|\leq
|R|$ and $|R^+|\le |R|-1$, since $x_h \in R$ has no following vertex
on $P$ and thus does not contribute an element to $R^+$.

According to Lemma \ref{lem-Posa},
 $$|N_G(R)| \leq |R^-\cup
R^+| \leq 2|R| - 1~,$$ and it follows that $|R|>k$. (Here the choice $\alpha=2$ in the definition of a $(k,\alpha)$-expander plays a crucial role.) Moreover,
$(x_0,v)$ is not an edge for any $v\in R$ (there is no $(h+1)$-cycle
in the graph), whereas adding any edge $(x_0,v)$ for $v\in R$
creates an $(h+1)$-cycle.

Fix a subset $\{y_1,\ldots,y_{k+1}\}\subset R$. For every $y_i\in
R$, there is a path $P_i$ ending at $y_i$, that can be obtained from
$P$ by a sequence of elementary rotations. Now fix $y_i$ as the
starting point of $P_i$ and let $Y_i$ be the set of endpoints of all
paths obtained from $P_i$ by a sequence of elementary rotations. As
before, $|Y_i|\ge k+1$, no edge joins $y_i$ to $Y_i$, and adding any
such edge creates a cycle of length $h+1$. Altogether we have found
$(k+1)^2$ pairs $(y_i,x_{ij})$ for $x_{ij}\in Y_i$. As every
non-edge is counted at most twice, the conclusion of the lemma
follows.
\end{proof}

\medskip

The reason we are after a cycle of length $h+1$ in the above
argument is that if $h+1=n$, then a Hamilton cycle is created.
Otherwise, if the graph is connected, then there will be an edge $e$
connecting a newly created cycle $C$ of length $h+1$ with a vertex
outside $C$. Then opening $C$ up and appending $e$ in an obvious way
creates a longer path in $G$.

In order to utilize quantitatively the above argument, we introduce
the notion of boosters.

\begin{defin}\label{booster}
Given a graph $G$, a non-edge $e=(u,v)$ of $G$ is called a {\em
booster} if adding $e$ to $G$ creates a graph $G'$, which is
Hamiltonian or whose longest path is longer than that of $G$.
\end{defin}
Note that technically every non-edge of a Hamiltonian graph $G$ is a
booster by definition.

Boosters advance a graph towards Hamiltonicity when added; adding
sequentially $n$ boosters clearly brings any graph on $n$ vertices
to Hamiltonicity.

We thus conclude from the previous discussion:
\begin{coro}\label{cor-Posa3}
Let $G$ be a connected non-Hamiltonian $(k,2)$-expander. Then $G$ has at
least $\frac{(k+1)^2}{2}$ boosters.
\end{coro}

\section{Long paths in random graphs}\label{s-longpaths}

In this section we treat the appearance of long paths and cycles in sparse random graphs. We will work with the probability space $G(n,p)$ of binomial random graphs, analogous results for the sister model $G(n,m)$ can be either proven using very similar arguments, or derived using available equivalence statements between the two models.

Our goal here is two-fold: we first prove that already in the super-critical regime $p=\frac{1+\epsilon}{n}$, the random graph $G(n,p)$ contains typically a path of length linear in $n$; then we prove that in the regime $p=\frac{C}{n}$, the random graph $G(n,p)$ has typically a path, covering the proportion of vertices tending to 1 as the constant $C$ increases. We will invoke the approaches and results developed in Section \ref{ss-DFS} (the Depth First Search Algorithm and its consequences) to achieve both of these goals, in fact in a rather short and elegant way.

\subsection{Linearly long paths in the supercritical regime}
In their groundbreaking paper \cite{ER60} from 1960, Paul Erd\H os
and Alfr\'ed R\'enyi made the following fundamental discovery: the
random graph $G(n,p)$ undergoes a remarkable phase transition around
the edge probability $p(n)=\frac{1}{n}$. For any constant
$\epsilon>0$, if $p=\frac{1-\epsilon}{n}$, then $G(n,p)$ has {\bf whp} all
connected components of size at most logarithmic in $n$, while for
$p=\frac{1+\epsilon}{n}$ {\bf whp} a connected component of linear size,
usually called the giant component, emerges in $G(n,p)$ (they also
showed that {\bf whp} there is a unique linear sized component). The
Erd\H os-R\'enyi paper, which launched the modern theory of random
graphs, has had enormous influence on the development of the field.
We will be able to derive both parts of this result very soon.

Although for the super-critical case $p=\frac{1+\epsilon}{n}$ the result of Erd\H{o}s and R\'enyi shows a typical existence of a linear sized connected component, it does not imply that a longest path in such a random graph is {\bf whp} linearly long. This was established some 20 years later by Ajtai, Koml\'os and Szemer\'edi \cite{AKS81}. In this section we present a fairly simple proof of their result.  We will not attempt to
achieve the best possible absolute constants, aiming rather for
simplicity. Our treatment follows closely that of \cite{KS13}.

The most fundamental idea of the proof is to run the DFS algorithm on a random graph $G\sim G(n,p)$, constructing the graph ``on the fly", as the algorithm progresses.  We first fix the
order $\pi$ on $V(G)=[n]$ to be the identity permutation. When
the DFS algorithm is fed with a sequence of i.i.d. Bernoulli($p$)
random variables $\bar{X}=(X_i)_{i=1}^N$, so that is gets its $i$-th
query answered positively if $X_i=1$ and answered negatively
otherwise, the so obtained graph is clearly distributed according to
$G(n,p)$. Thus, studying the component structure of $G$ can be
reduced to studying the properties of the random sequence $\bar{X}$. This is a very useful trick, as it allows to ``flatten" the random graph by replacing an inherently two-dimensional structure (a graph) by a one-dimensional one (a sequence of bits).
In particular (here and later we use the DFS algorithm-related notation of Section \ref{ss-DFS}), observe crucially that as long as $T\ne\emptyset$,
every positive answer to a query results in a vertex being moved
from $T$ to $U$, and thus after $t$ queries and assuming
$T\ne\emptyset$ still, we have $|S\cup U|\ge \sum_{i=1}^t X_i$. (The
last inequality is strict in fact as the first vertex of each
connected component is moved from $T$ to $U$ ``for free", i.e.,
without need to get a positive answer to a query.) On the other
hand, since the addition of every vertex, but the first one in a
connected component, to $U$ is caused by a positive answer to a
query, we have at time $t$: $|U|\le 1+\sum_{i=1}^t X_i$.

The probabilistic part of our argument is provided by the following
quite simple lemma.

\begin{lemma}\label{le3-1}
Let $\epsilon>0$ be a small enough constant. Consider the sequence
$\bar{X}=(X_i)_{i=1}^N$ of i.i.d. Bernoulli random variables with
parameter $p$.
\begin{enumerate}
\item Let $p=\frac{1-\epsilon}{n}$.  Let $k=\frac{7}{\epsilon^2}\ln
n$. Then {\bf whp} there is no interval of length $kn$ in $[N]$, in which
at least $k$ of the random variables $X_i$ take value 1.
\item Let $p=\frac{1+\epsilon}{n}$. Let $N_0=\frac{\epsilon
n^2}{2}$. Then {\bf whp} $\left|\sum_{i=1}^{N_0} X_i
-\frac{\epsilon(1+\epsilon)n}{2}\right|\le n^{2/3}$.
\end{enumerate}
\end{lemma}

\begin{proof} {1)} For a given interval $I$ of length $kn$ in $[N]$, the
sum $\sum_{i\in I} X_i$ is distributed binomially with parameters
$kn$ and $p$. Applying the  Chernoff bound to the upper tail of $B(kn,p)$, and
then the union bound, we see that the probability of the existence
of an interval violating the assertion of the lemma is at most
$$
(N-k+1)Pr[B(kn,p)\ge k]< n^2\cdot
e^{-\frac{\epsilon^2}{3}(1-\epsilon)k} <n^2\cdot
e^{-\frac{\epsilon^2(1-\epsilon)}{3}\,\frac{7}{\epsilon^2}\ln n}=
o(1)\,,
$$
for small enough $\epsilon>0$.

\noindent{2)} The sum $\sum_{i=1}^{N_0}X_i$ is distributed
binomially with parameters $N_0$ and $p$. Hence, its expectation is
$N_0p=\frac{\epsilon n^2p}{2}=\frac{\epsilon(1+\epsilon)n}{2}$, and
its standard deviation is of order $\sqrt{n}$. Applying the Chebyshev
inequality, we get the required estimate.
\end{proof}

\medskip

Now we are ready to formulate and to prove the main result of this section.

\begin{theo}\label{th1}
Let $\epsilon>0$ be a small enough constant. Let $G\sim G(n,p)$.
\begin{enumerate}
\item Let $p=\frac{1-\epsilon}{n}$. Then {\bf whp} all connected
components of $G$ are of size at most $\frac{7}{\epsilon^2}\ln n$.
\item Let $p=\frac{1+\epsilon}{n}$. Then {\bf whp} $G$ contains a path on
at least $\frac{\epsilon^2 n}{5}$ vertices.
\end{enumerate}
\end{theo}
In both cases, we run the DFS algorithm on $G\sim G(n,p)$, and
assume that the sequence $\bar{X}=(X_i)_{i=1}^N$ of random
variables, defining the random graph $G\sim G(n,p)$ and guiding the
DFS algorithm, satisfies the corresponding part of Lemma \ref{le3-1}.

\medskip

\begin{proof} 1) Assume to the contrary that $G$ contains a connected
component $C$ with more than $k=\frac{7}{\epsilon^2}\ln n$ vertices.
Let us look at the epoch of the DFS when $C$ was created. Consider
the moment inside this epoch when the algorithm has found the
$(k+1)$-st vertex of $C$ and is about to move it to $U$. Denote
$\Delta S=S\cap C$ at that moment. Then $|\Delta S\cup U|=k$, and
thus the algorithm got exactly $k$ positive  answers to its queries
to random variables $X_i$ during the epoch, with each positive
answer being responsible for revealing a new vertex of $C$, after
the first vertex of $C$ was put into $U$ in the beginning of the
epoch. At that moment during the epoch only pairs of edges touching
$\Delta S\cup U$ have been queried, and the number of such pairs is
therefore at most $\binom{k}{2}+k(n-k)<kn$. It thus follows that the
sequence $\bar{X}$ contains an interval of length at most $kn$ with
at least $k$ 1's inside --- a contradiction to Property 1 of Lemma
\ref{le3-1}.

2) Assume that the sequence $\bar{X}$ satisfies Property 2 of Lemma
\ref{le3-1}. We claim that after the first $N_0=\frac{\epsilon
n^2}{2}$ queries of the DFS algorithm, the set $U$ contains at least
$\frac{\epsilon^2 n}{5}$ vertices (with the contents of $U$ forming
a path of desired length at that moment). Observe first that
$|S|<\frac{n}{3}$ at time $N_0$. Indeed, if $|S|\ge\frac{n}{3}$,
then let us look at a moment $t\le N_0$ where $|S|=\frac{n}{3}$ (such a
moment surely exists as vertices flow to $S$ one by one). At that
moment $|U|\le 1+\sum_{i=1}^{t}X_i<\frac{n}{3}$ by Property 2 of
Lemma \ref{le3-1}. Then $|T|=n-|S|-|U|\ge \frac{n}{3}$, and the
algorithm has examined all $|S|\cdot|T|\ge \frac{n^2}{9}>N_0$ pairs
between $S$ and $T$ (and found them to be non-edges) --- a
contradiction. Let us return to time $N_0$. If $|S|<\frac{n}{3}$ and
$|U|<\frac{\epsilon^2 n}{5}$ then, we have $T\ne\emptyset$. This
means in particular that the algorithm is still revealing the
connected components of $G$, and each positive answer it got
resulted in moving a vertex from $T$ to $U$ (some of these vertices
may have already migrated further from $U$ to $S$). By Property 2 of
Lemma \ref{le3-1} the number of positive answers at that point is at
least $\frac{\epsilon(1+\epsilon)n}{2}-n^{2/3}$. Hence we have
$|S\cup U|\ge \frac{\epsilon(1+\epsilon)n}{2}-n^{2/3}$. If
$|U|\le\frac{\epsilon^2n}{5}$, then $|S|\ge \frac{\epsilon
n}{2}+\frac{3\epsilon^2n}{10}-n^{2/3}$. All $|S||T|\ge
|S|\left(n-|S|-\frac{\epsilon^2n}{5}\right)$ pairs between $S$ and
$T$ have been probed by the algorithm (and answered in the
negative). We thus get:
\begin{eqnarray*}
\frac{\epsilon n^2}{2}&=&N_0\ge
|S|\left(n-|S|-\frac{\epsilon^2n}{5}\right)\ge \left(\frac{\epsilon
n}{2}+\frac{3\epsilon^2n}{10}-n^{2/3}\right) \left(n-\frac{\epsilon
n}{2}-\frac{\epsilon^2n}{2}+n^{2/3}\right)\\
&=&\frac{\epsilon n^2}{2}+\frac{\epsilon^2n^2}{20}-O(\epsilon^3)n^2
> \frac{\epsilon n^2}{2}
\end{eqnarray*}
(we used the assumption $|S|<\frac{n}{3}$ in the second inequality above), and this is obviously a
contradiction, completing the proof.
\end{proof}

\medskip

Let us discuss briefly the obtained result and its proof. First, given the probable existence of a long path in $G(n,p)$, that of a long cycle is just one short step further. Indeed, we can use sprinkling as follows. Let $p=\frac{1+\epsilon}{n}$ for small $\epsilon>0$. Write $1-p=(1-p_1)(1-p_2)$ with $p_2=\frac{\epsilon}{2n}$; thus, most of the probability $p$ goes into $p_1\ge \frac{1+\epsilon/2}{n}$. Let now $G\sim G(n,p)$, $G_1\sim G(n,p_1)$, $G_2\sim G(n,p_2)$, we can represent $G=G_1\cup G_2$. By Theorem \ref{th1}, $G_1$ {\bf whp} contains a linearly long path $P$. Now, the edges of $G_2$ can be used to close most of $P$ into a cycle -- there is {\bf whp} an edge of $G_2$ between the first and the last $n^{2/3}$ (say) vertices of $P$.

The dependencies on $\epsilon$ in both parts of Theorem \ref{th1} are of the correct order of
magnitude -- for $p=\frac{1-\epsilon}{n}$ a largest connected
component of $G(n,p)$ is known to be {\bf whp} of size
$\Theta(\epsilon^{-2})\log n$
while for $p=\frac{1+\epsilon}{n}$ a longest cycle
of $G(n,p)$ is {\bf whp} of length $\Theta(\epsilon^2)n$
(see, e.g., Chapter 6 of \cite{Bol-book}).

Observe that using a Chernoff-type bound for the
tales of the binomial random variable instead of the Chebyshev
inequality would allow us to claim in the second part of Lemma
\ref{le3-1} that the sum $\sum_{i=1}^{N_0} X_i$ is close to
$\frac{\epsilon(1+\epsilon)n}{2}$ with probability exponentially
close to 1. This would show in turn, employing the argument of
Theorem \ref{th1}, that $G(n,p)$ with $p=\frac{1+\epsilon}{n}$
contains a path of length linear in $n$ with exponentially high
probability, namely, with probability $1-\exp\{-c(\epsilon)n\}$.

 As we have  mentioned in Section \ref{ss-DFS}, the DFS algorithm is
applicable equally well to directed graphs. Hence essentially the
same argument as above, with obvious minor changes, can be applied
to the model $D(n,p)$ of random digraphs. It then yields the following theorem:
\begin{theo}\label{thm2}
Let $p=\frac{1+\epsilon}{n}$, for $\epsilon>0$ constant. Then the
random digraph $D(n,p)$ has {\bf whp} a directed path and a directed
cycle of length $\Theta(\epsilon^2)n$.
\end{theo}
This recovers the result of Karp \cite{Kar90}.

\subsection{Nearly spanning paths}

Consider now the regime $p=\frac{C}{n}$, where $C$ is a (large) constant. Our goal is to prove that {\bf whp} in $G(n,p)$, the length of a longest path approaches $n$ as $C$ tends to infinity. This too is a classical result due to Ajtai, Koml\'os and Szemer\'edi \cite{AKS81}, and independently due to Fernandez de la Vega \cite{Vega79}. It is fairly amusing to see how easily it can be derived using the DFS-based tools we developed in Section \ref{ss-DFS}.

\begin{theo}\label{t21}
For every $\epsilon>0$  there exists $C=C(\epsilon)>0$ such that the following is true.
Let $G\sim G(n,p)$, where $p=\frac{C}{n}$. Then {\bf whp} $G$ contains a path of length at least $(1-\epsilon)n$.
\end{theo}

\begin{proof}
Clearly we can assume $\epsilon>0$ to be small enough. Let $k=\lfloor\frac{\epsilon n}{2}\rfloor$. By Proposition \ref{DFS2} it suffices to show that $G\sim G(n,p)$ contains {\bf whp} an edge between every pair of disjoint subsets of size $k$ of $V(G)$. For a given pair of disjoint sets $S,T$ of size $|S|=|T|=k$, the probability that $G$ contains no edges between $S$ and $T$ is exactly $(1-p)^{k^2}$ (all $k^2$ pairs between $S$ and $T$ come out non-edges in $G$). Using the union bound, we obtain that the probability of the existence of a pair violating this requirement is at most
$$
\binom{n}{k}\binom{n-k}{k}(1-p)^{k^2}<\binom{n}{k}^2(1-p)^{k^2}<\left(\frac{en}{k}\right)^{2k}e^{-pk^2}
< \left[\left(\frac{en}{k}\right)^2\cdot e^{-\frac{Ck}{n}}\right]^k\ .
$$
Recalling the value of $k$ and taking $C=\frac{5\ln(1/\epsilon)}{\epsilon}$ guarantees that the above estimate vanishes (in fact exponentially fast) in $n$, thus establishing the claim.
\end{proof}

\medskip

As before, getting a nearly spanning cycle {\bf whp} can be easily done through sprinkling.

\bigskip

A similar statement holds for the probability space $D(n,p)$ of random directed graphs, with an essentially identical proof.

\section{The appearance of Hamilton cycles in random graphs}\label{s-Ham}

The main aim of this section is to establish the Hamiltonicity threshold in the probability space $G(n,p)$, this  is the minimum value of the edge probability $p(n)$, for which a random graph $G$ drawn from $G(n,p)$ is {\bf whp} Hamiltonian. By doing so we will prove a classical result of Koml\'os and Szemer\'edi \cite{KS83} and independently of Bollob\'as \cite{B84}.

Let us start by providing some intuition on where this threshold is expected to be located. It is well known that the threshold probability for connectivity in $G(n,p)$ is $p=\frac{\ln n}{n}$. More explicitly, one can prove (and we leave this as an exercise) that for any function $\omega(n)$ tending to infinity arbitrarily slowly with $n$, if $p=\frac{\ln n-\omega(n)}{n}$, then {\bf whp} $G\sim G(n,p)$ is not connected, whereas for $p=\frac{\ln n+\omega(n)}{n}$ {\bf whp} $G\sim G(n,p)$ is connected. Perhaps more importantly, the main reason for the threshold for connectivity to be around $\ln n/n$ is that precisely at this value of probability the last isolated vertex in $G(n,p)$ typically ceases to exist. Of course, the graph cannot be connected while having isolated vertices, and this is the easy part of the connectivity threshold statement; the hard(er) part is to prove that if $p(n)$ is such that $\delta(G)\ge 1$ {\bf whp}, then $G$ is {\bf whp} connected.

If so, we can suspect that the threshold for Hamiltonicity of $G(n,p)$ coincides with that of non-existence of vertices of degree at most one, the latter being an obvious necessary condition for Hamiltonicity. This is exactly what was proven in \cite{B84,KS83}. Let us therefore set our goal by stating first a fairly accessible result about the threshold for $\delta(G)\ge 2$, both in $G(n,p)$ and in $G(n,m)$.

\begin{prop}\label{mindeg2}
Let $\omega(n)$ be any function tending to infinity arbitrarily slowly with $n$. Then:
\begin{itemize}
\item in the probability space $G(n,p)$,
  \begin{enumerate}
   \item if $p(n)=\frac{\ln n+\ln\ln n-\omega(n)}{n}$, then $G\sim G(n,p)$ {\bf whp} satisfies $\delta(G)\le 1$;
   \item if $p(n)=\frac{\ln n+\ln\ln n+\omega(n)}{n}$, then $G\sim G(n,p)$ {\bf whp} satisfies $\delta(G)\ge 2$;
  \end{enumerate}
\item in the probability space $G(n,m)$,
  \begin{enumerate}
   \item if $m(n)=\frac{(\ln n+\ln\ln n-\omega(n))n}{2}$, then $G\sim G(n,m)$ {\bf whp} satisfies $\delta(G)\le 1$;
   \item if $m(n)=\frac{(\ln n+\ln\ln n+\omega(n))n}{2}$, then $G\sim G(n,m)$ {\bf whp} satisfies $\delta(G)\ge 2$.
  \end{enumerate}
\end{itemize}
\end{prop}

\begin{proof} Straightforward application of the first (for proving $\delta(G)\ge 2$) and the second (for proving $\delta(G)\le 1$) moment methods in both probability spaces; left as an exercise.
\end{proof}

\medskip

Hence our goal will be to prove that for $p(n)=\frac{\ln n+\ln\ln n+\omega(n)}{n}$ and for $m(n)=\frac{(\ln n+\ln\ln n+\omega(n))n}{2}$ the random graphs $G(n,p)$ and $G(n,m)$ respectively are {\bf whp} Hamiltonian.

We will actually prove a stronger, and a much more delicate, result about the hitting time of Hamiltonicity in random graph processes. Let us first define this notion formally. Let $\sigma: E(K_n)\rightarrow [N]$ be a permutation of the edges of the complete graph $K_n$ on $n$ vertices, we can write $\sigma=(e_1,\ldots,e_N)$, where $N=\binom{n}{2}$. A {\em graph process} $\tilde{G}=\tilde{G}(\sigma)$ is a nested sequence $\tilde{G}=(G_i)_{i=0}^N$, where the graph $G_i$ has $[n]$ as its vertex set and $\{e_1,\ldots, e_i\}$, the prefix of $\sigma$ of length $i$, as its edge set. The sequence $(G_i)$ thus starts with the empty graph on $n$ vertices, finishes with the complete graph on $n$ vertices, and its $i$-th element $G_i$ has exactly $i$ edges; moreover, it is nested, as for $i\ge 1$ the graph $G_{i}$ is obtained from its predecessor $G_{i-1}$ by adding the $i$-th edge $e_{i}$ of $\sigma$. We can view $\tilde{G}(\sigma)$ as a graph process (as the name indicates suggestively) or as an evolutionary process, unraveling from the empty graph to the complete graph, as guided by $\sigma$.

Now, we introduce the element of randomness in the above definition. Suppose the permutation $\sigma$ is drawn uniformly at random from the set of all $N!$ permutations of the edges of $K_n$.  Then the corresponding process $\tilde{G}(\sigma)$ is called a {\em random graph process}. We can describe it in the following equivalent way: start by setting $G_0$ to be the empty graph on $n$ vertices, and for each $1\le i\le N$, obtain $G_i$ by choosing an edge $e_i$ of $K_n$ missing in $G_{i-1}$ uniformly at random and adding it to $G_{i-1}$. This very nice and natural probability space models a random evolutionary process in graphs; here too we proceed from the empty graph to the complete graph, but in a random fashion.

Random graph processes are so important not just because they model evolution very nicely; in fact, they embed the probability spaces $G(n,m)$ for various $m$; due to standard connections between $G(n,m)$ and $G(n,p)$ one can also claim they ``contain" $G(n,p)$ as well. Observe that running a random process $\tilde{G}$ and stopping it (or taking a {\em snapshot}) at time $m$ produces the probability distribution $G(n,m)$. Indeed, every graph $G$ with vertex set $[n]$ and exactly $m$ edges is the $m$-th element of the same number of graph processes, namely, of $m!(N-m)!$ of them. Thus, understanding random graph processes usually leads to immediate consequences for $G(n,m)$, and then for $G(n,p)$, and Hamiltonicity is not exceptional in this sense.

Let ${\cal P}$ be a property of graphs on $n$ vertices; assume that $P$ is monotone increasing (i.e., adding edges preserves it), and that  the complete graph $K_n$ possesses ${\cal P}$ (one can think of ${\cal P}$ as being the property of Hamiltonicity). Then, given a permutation $\sigma: E(K_n)\rightarrow [N]$ and the corresponding graph process $\tilde{G}(\sigma)$, we can define the first moment $i$
when the $i$-th element $G_i$ of $\tilde{G}$ has ${\cal P}$. This is the so called {\em hitting time} of ${\cal P}$, denoted by $\tau_{{\cal P}}(\tilde{G}(\sigma))$:
$$
\tau_{{\cal P}}(\tilde{G}(\sigma))=\min\{i\ge 0: G_i\mbox{ has } {\cal P}\}\,.
$$
Of course, due to the monotonicity of ${\cal P}$ from this point till the end of the process the graphs $G_i$ all have ${\cal P}$. When $\tilde{G}$ is a random graph process, the hitting time $\tau_{{\cal P}}(\tilde{G})$ becomes a random variable, and one can study its typical behavior. A related task is to compare two hitting times, and to try to bundle them, deterministically or probabilistically.

We now state the main result of this section, due to Ajtai, Koml\'os and Szemeredi \cite{AKS85}, and to Bollob\'as \cite{B84}.

\begin{theo}\label{th-Ham-hit}
Let $\tilde{G}$ be a random graph process on $n$ vertices. Denote by $\tau_2(\tilde{G})$ and $\tau_{{\cal H}}(\tilde{G})$ the hitting times of the properties of having minimum degree at least 2, and of Hamiltonicity, respectively. Then {\bf whp}:
$$
\tau_2(\tilde{G})=\tau_{{\cal H}}(\tilde{G})\,.
$$
\end{theo}

\medskip

In words, for a typical graph process, Hamiltonicity arrives {\em exactly} at the very moment the last vertex of degree less than two disappears. Of course, it cannot arrive earlier deterministically, so the main point of the above theorem is to prove that typically is does not arrive later either.

As we indicated above, random graph process results are usually more powerful than those for concrete random graph models. Here too we are able to derive the results for $G(n,m)$ and $G(n,p)$ easily from the above theorem.

\begin{coro}\label{Ham-Gnm}
Let $m(n)=\frac{(\ln n+\ln\ln n+\omega(n))n}{2}$. Then a random graph $G\sim G(n,m)$ is {\bf whp}
 Hamiltonian.
 \end{coro}

 \begin{proof}
Generate a random graph $G$ distributed according to $G(n,m)$ by running a random graph process $\tilde{G}$ and stopping it at time $m$. By Proposition \ref{mindeg2} we know that $\tau_2(\tilde{G})\le m$.  Theorem \ref{th-Ham-hit} implies that typically $\tau_2(\tilde{G})=\tau_{{\cal H}}(\tilde{G})$, and thus the graph of the process has become Hamiltonian not later than $m$. Hence $G$ is {\bf whp} Hamiltonian as well.
\end{proof}

\medskip

\begin{coro}\label{Ham-Gnp}
Let $p(n)=\frac{\ln n+\ln\ln n+\omega(n)}{n}$. Then a random graph $G\sim G(n,m)$ is {\bf whp}
 Hamiltonian.
 \end{coro}

\begin{proof}
Observe that generating a random graph $G\sim G(n,p)$ and conditioning on its number of edges being exactly equal to $m$ produces the distribution $G(n,m)$.
Let $\omega_1(n)=\omega(n)/3$. Denote $I=[Np-n\omega_1(n),Np+n\omega_1(n)]$. Observe that for every $m\in I$, the random graph $G\sim G(n,m)$ is {\bf whp} Hamiltonian by Corollary \ref{Ham-Gnm}. Also, the number of edges in $G(n,p)$ is distributed binomially with parameters $N$ and $p$ and has thus standard deviation less than $\sqrt{Np}\ll n\omega_1(n)$. Applying Chebyshev we derive that {\bf whp} $|E(G)|\in I$. Hence
\begin{gather*}
Pr[G\sim G(n,p)\mbox{ is not Hamiltonian}]=\sum_{m=0}^N Pr[|E(G)|=m]\cdot Pr[G\mbox{ is not Hamiltonian}|\, |E(G)=m]\\
\le Pr[|E(G)|\not\in I] +\sum_{m\in I} Pr[|E(G)|=m]\,Pr[G\mbox{ is not Hamiltonian}|\, |E(G)=m]\\
= o(1) + \sum_{m\in I} Pr[|E(G)|=m]\,Pr[G\sim G(n,m)\mbox{ is not Hamiltonian}|=
o(1)\cdot Pr[ Bin(N,p)\in I]\\
=o(1)\,.
\end{gather*}
\end{proof}

\medskip

Now we start proving Theorem \ref{th-Ham-hit}. The proof is somewhat technical, so before diving into its details, we outline its main idea briefly. Recall that our goal is to prove that for a typical random process $\tilde{G}$, we have $\tau_2(\tilde{G})=\tau_{{\cal H}}(\tilde{G})$. In order to prove this, we will take a very close look at the snapshot $G_{\tau_2}$ of $\tilde{G}$, aiming to prove that this graph is {\bf whp} Hamiltonian. By definition, the minimum degree of $G_{\tau_2}$ is exactly two, and it is thus quite reasonable to expect that this graph is typically a $(k,2)$-expander for $k=\Theta(n)$. This is true indeed, however such expansion by itself does not quite guarantee Hamiltonicity. As indicated in Section \ref{ss-boosters}, expanders form a very convenient backbone for augmenting a graph to a Hamiltonian one --- according to Corollary \ref{cor-Posa3} every connected non-Hamiltonian $(k,2)$-expander has $\Omega(k^2)$ boosters. Observe though that since we aim to prove a hitting time result, we cannot allow ourselves to sprinkle few random edges on top of our expander --- a Hamilton cycle should appear at the very moment the minimum degree in the random graph process becomes two. We will circumvent this difficulty in the following way: we will argue that the snapshop $G_{\tau_2}$ typically is not only a good expander by itself, but also contains a subgraph
$\Gamma_0$ which is about as good an expander as $G_{\tau_2}$ is, but contains only a small positive proportion of its edges. Having obtained such $\Gamma_0$, we will start looking for boosters relative to $\Gamma_0$, but already contained in our graph $G_{\tau_2}$ -- thus avoiding the need for sprinkling. We will argue that $G_{\tau_2}$ is typically such that it is contains a booster with respect to every sparse expander in it. If this is the case, then we will be able to start with $\Gamma_0$ and to update it sequentially by adding a booster after a booster  (at most $n$ boosters will need to be added by definition), until we will finally reach Hamiltonocity -- all within $G_{\tau_2}$; observe crucially that at each step of this augmentation procedure the updated backbone $\Gamma_i$, obtained by adjoining to $\Gamma_0$ the previously added boosters, has at most $n$ more edges than $\Gamma_0$ and is thus still a sparse subgraph of $G_{\tau_2}$; of course the required expansion is inherited from an iteration to iteration. Then our claim about $G_{\tau_2}$ typically containing a booster with respect to every sparse expander within is applicable, and we can push the process through. This is quite a peculiar proof idea -- it appears that the random graph is helping itself to become Hamiltonian!

Let us get to work. As outlined before, we run a random graph process $\tilde{G}$ and take a snapshot at the hitting time $\tau_2=\tau_2(\tilde{G})$. Denote
\begin{eqnarray*}
m_1&=&\frac{n\ln n}{2}\,\\
m_2&=&n\ln n\,.
\end{eqnarray*}
Observe that by Proposition \ref{mindeg2} we have that {\bf whp} $m_1\le \tau_2\le m_2$. Let
$$
d_0=\lfloor \delta_0 \ln n\rfloor\,
$$
where $\delta_0>0$ is a sufficiently small constant to be chosen later, and denote, for a graph $G$ on $n$ vertices,
$$
SMALL(G)=\{v\in V(G): d(v)<d_0\}\,.
$$
Observe that for $G\sim G(n,m)$ with $m\ge m_1$, the expected vertex degree is asymptotically equal to $\ln n$. Thus falling into $SMALL(G)$ is a rather rare event, and we can expect the vertices of $SMALL(G)$ to be few and far apart in the graph. In addition, such $G$ should typically have a very nice edge distribution, with no small and dense vertex subsets, and many edges crossing between any two large disjoint subsets. This is formalized in the following lemma.

\begin{lemma}\label{le4-1}
Let $\tilde{G}=(G_i)_{i=0}^N$ be a random graph process on $n$ vertices. Denote $G=G_{\tau_2}$, where $\tau_2=\tau_2(\tilde{G})$ is the hitting time for having minimum degree two in $\tilde{G}$. Then {\bf whp} $G$ has the following properties:
\begin{itemize}
\item[{\bf (P1)}] $\Delta(G)\le 10\ln n$; $\delta(G)\ge 2$;
\item[{\bf (P2)}] $|SMALL(G)|\le n^{0.3}$;
\item[{\bf (P3)}] $G$ does not contain a non-empty path of length at most 4  such that both of its (possibly identical) endpoints lie in $SMALL(G)$;
\item[{\bf (P4)}] every vertex subset $U\subset [n]$ of size $|U|\le \frac{n}{\ln^{1/2}n}$ spans at most $|U|\cdot\ln^{3/4}n$ edges in $G$;
\item[{\bf (P5)}] for every pair of disjoint vertex subsets $U,W$ of sizes  $|U|\le \frac{n}{\ln^{1/2}n}$, $|W|\le |U|\cdot\ln^{1/4}n$, the number of edges of $G$ crossing between $U$ and $W$ is at most $\frac{d_0|U|}{2}$;
\item[{\bf (P6)}] for every pair of disjoint vertex subsets $U,W$ of size $|U|=|W|=\left\lceil\frac{n}{\ln^{1/2}n}\right\rceil$, $G$ has at least $0.5n$ edges between $U$ and $W$.
\end{itemize}
\end{lemma}

\begin{proof}
The proof is basically a fairly standard (though tedious) manipulation with binomial coefficients. We will thus prove several of the above items, leaving the proof of remaining ones to the reader.

\medskip

\noindent{\bf (P1)}: Observe that since {\bf whp} $\tau_2\le m_2$, it is enough to prove that in $G\sim G(n,m_2)$ there are {\bf whp} no vertices of degree at least $10\ln n$. For a given vertex $v\in [n]$, the probability that $v$ has degree at least $10\ln n$ in $G(n,m_2)$ is at most
$$
\binom{n-1}{10\ln n}\frac{\binom{N-10\ln n}{m_2-10\ln n}}{\binom{N}{m_2}}
\le \left(\frac{en}{10\ln n}\right)^{10\ln n}\, \left(\frac{m_2}{N}\right)^{10\ln n}\,,
$$
by the standard estimates on binomial coefficients stated in Section \ref{sss-asymp}. After cancellations we see that the above estimate is at most $(en/5(n-1))^{10\ln n}=o(1/n)$. Applying the union bound we obtain that typically at time $m_2$, and thus at $\tau_2\le m_2$ as well, there are no vertices of degree at least $10\ln n$. The bound on $\delta(G)$ is immediate from the definition of $\tau_2$.

\medskip

\noindent{\bf (P2)}: Notice that since adding edges can only decrease the size of $SMALL(G)$, it is enough to prove that typically already at time $m_1$ $|SMALL(G_{m_1})|\le n^{0.3}$. Let $G\sim G(n,m_1)$. If $|SMALL(G)|\ge n^{0.3}$, then $G$ contains a subset $V_0\subset V$, $|V_0|=k=\lceil n^{0.3}\rceil$ such that $e_G(V_0,V-V_0)\le d_0k$. The probability of this to happen in $G(n,m_1)$ is at most:
\begin{gather*}
\binom{n}{k}\sum_{i\le d_0k}\binom{k(n-k)}{i}\cdot\frac{\binom{N-k(n-k)}{m_1-i}}{\binom{N}{m_1}}
\le \binom{n}{k}\sum_{i\le d_0k}\binom{kn}{i}\cdot\frac{\binom{N-k(n-k)}{m_1-i}}{\binom{N-i}{m_1-i}}\cdot
\frac{\binom{N-i}{m_1-i}}{\binom{N}{m_1}}\\
\le \left(\frac{en}{k}\right)^k\sum_{i\le d_0k} \left(\frac{ekn}{i}\right)^i\cdot
e^{-\frac{(m_1-i)(k(n-k)-i)}{N-i}}\cdot\left(\frac{m_1}{N}\right)^i
\le  \left(\frac{en}{k}\right)^k\sum_{i\le d_0k} \left(\frac{ekm_1n}{iN}\right)^i\cdot e^{-\frac{0.9km_1n}{N}}\\
\le \left(\frac{en}{k}\right)^k \cdot (d_0k+1)\cdot\left(\frac{ekm_1n}{d_0kN}\right)^{d_0k}\cdot  e^{-\frac{0.9km_1n}{N}}\\
\le (d_0k+1)\left[3n^{0.7}\left(\frac{3\ln n}{d_0}\right)^{d_0}\cdot e^{-0.8\ln n}\right]^k=o(1)\,,
\end{gather*}
for $\delta_0$ small enough.

\medskip

\noindent{\bf (P3)}: Since {\bf whp} $m_1\le \tau_2\le m_2$, it is enough to prove the following statement:{\bf whp} every two (possibly identical) vertices of $SMALL(G_{m_1})$ are not connected by a path of length at most 4 in $G_{m_2}$.

Let us prove first that {\bf whp} there is no such path in $G_{m_1}\sim G(n,m_1)$. We start with the case where the endpoints of the path are distinct. Fix $1\le r\le 4$, a sequence $P$ of distinct vertices $v_0,\ldots,v_r$ in $[n]$ and denote by ${\cal A}_P$ the event $(v_i,v_{i+1})\in E(G_{m_1}$ for every $0\le i\le r-1$. Then
$$
Pr[{\cal A}_P]=\frac{\binom{N-r}{m_1-r}}{\binom{N}{m_1}}\le\left(\frac{m_1}{N}\right)^r=\left(\frac{\ln n}{n-1}\right)^r\,.
$$
If we now condition on ${\cal A}_P$, then the two edges $(v_0,v_1)$ and $(v_{r-1},v_r)$ are present in $G_{m_1}$. Thus in order for both $v_0,v_r$ to fall into $SMALL(G_{m_1})$, out of $2n-4$ potential edges between $\{v_0,v_r\}$ and the rest of the graph (the edges $(v_0,v_1),(v_{r-1},v_r)$ are excluded from the count), only at most $2d_0-2$ are present in $G_{m_1}$. Hence:
\begin{eqnarray*}
Pr[v_0,v_r\in SMALL(G_{m_1}) | {\cal A}_P] &\le& \sum_{i=0}^{2d_0-2}\binom{2n-4}{i}\cdot
\frac{\binom{N-r-2n+4}{m_1-r-i}}{\binom{N-r}{m_1-r}}\\
&\le& (2d_0-1)\binom{2n-4}{2d_0-2}\cdot \frac{\binom{N-r-2n+4}{m_1-r-2d_0+2}}{\binom{N-r}{m_1-r}}\\
&\le&  2d_0 \binom{2n-4}{2d_0-2}\cdot \frac{\binom{N-r-2n+4}{m_1-r-2d_0+2}}{\binom{N-r-2d_0+2}{m_1-r-2d_0+2}}\cdot
\frac {\binom{N-r-2d_0+2}{m_1-r-2d_0+2}}{\binom{N-r}{m_1-r}}\\
&\le& 2d_0\cdot \left(\frac{en}{d_0-1}\right)^{2d_0-2}\cdot e^{-\frac{(m_1-r-2d_0+2)(2n-2d_0-2)}{N-r-2d_0+2}}\cdot
\left(\frac{m_1-r}{N-r}\right)^{2d_0-2}\\
&\le& 2d_0\cdot \left(\frac{em_1n}{(d_0-1)N}\right)^{2d_0-2}\cdot e^{-\frac{1.9m_1n}{N}}\le n^{-1.8}\,,
\end{eqnarray*}
for $\delta_0$ small enough. Hence, applying the union bound over all such sequences of $r+1$ vertices, we conclude that the probability that there exists a path in $G_{m_1}$ of length at most 4, connecting two distinct vertices from $SMALL(G_{m_1})$ is at most $\sum_{r\le 4} n^{r+1}\cdot \left(\frac{\ln n}{n-1}\right)^r\cdot n^{-1.8}=o(1)$. The case where the endpoitns of the path are identical is treated similarly.

In light of the above, we can assume that after $m_1$ steps of the random graph process the current graph does not have a forbidden short path between the vertices of $SMALL$. Moreover, by {\bf (P2)} we can assume that $|SMALL(G_{m_1})|\le n^{0.3}$. Now, let us run the process between $m_1$ and $m_2$. In order for the $i$-th edge of the process, $m_1<i\le m_2$, to close a short path between the vertices of $SMALL(G_{m_1})$, it should fall inside a current set $U$ of vertices at distance at most 3 from $SMALL(G_{m_1})$. We have proven (property {\bf (P1)}) that in fact ${\bf whp}$ the maximum degree of $G_{m_2}$ as well is at most $10\ln n$. Hence, ${\bf whp}$ in this time interval, the set $U$ has size at most $|SMALL(G_{m_1}|\cdot (10\ln n)^3$, and thus the probability of the $i$-th edge of the process to fall inside $U$ is at most $\frac{\binom{|U|}{2}}{N-m_2}=o(n^{-1.3})$. Taking the union bound over all such $i$ in the interval $(m_1,m_2]$, we establish the desired property.

Properties {\bf (P4)}--{\bf (P6)} can be proven quite similarly (and are in fact simpler to prove), and we spare the reader from the (perhaps somewhat boring\ldots) task of reading their proofs.

\end{proof}

\medskip

The above stated properties {\bf (P1)}--{\bf (P6)} are sufficient to prove that $G_{\tau_2}$ is a very good expander by itself. Our goal is somewhat different though -- we aim to prove that $G_{\tau_2}$ contains a much sparser, but still fairly good expander. For this purpose, assume that a graph $G=(V,E)$ has properties {\bf (P1)}--{\bf (P6)}. Form a random subgraph $\Gamma_0$ of $G$ as follows. For every $v\in V-SMALL(G)$, choose a set $E(v)$ of $d_0$ edges of $G$ incident to $v$ uniformly at random; for every $v\in SMALL(G)$, define $E(v)$ to be the set of {\em all} edges of $G$ touching $v$. Finally, define $\Gamma_0$ to be the spanning subgraph of $G$, whose edge set is:
$$
E(\Gamma_0)=\bigcup_{v} E(v)\,.
$$
In words, in order to form $\Gamma_0$ we retain all edges touching the vertices of $SMALL(G)$, and sparsify randomly other edges.

\begin{lemma}\label{sparse-exp}
With high probability (over the choices of $E(v)$) the subgraph $\Gamma_0$ is a $(k,2)$-expander with at most $d_0n$ edges, where $k=\frac{n}{4}$.
\end{lemma}

\begin{proof}
Since by definition $|E(v)|\le d_0$ for every $v\in V$, it follows immediately that $|E(\Gamma_0|\le d_0n$.
We now prove that typically $\Gamma_0$ has the following property:
\begin{itemize}
\item[{\bf (P7)}] For every pair of disjoint sets $U,W$ of size $|U|=|W|=\left\lceil\frac{n}{\ln^{1/2}n}\right\rceil$, $\Gamma_0$ has at least one edge between $U$ and $W$.
\end{itemize}

Fix sets $U,W$ as above. We know by {\bf (P6)} that $G$ has at least $0.5n$ edges between $U$ and $W$. For a vertex $u\in U$, the probability that none of the edges between $u$ and $W$ falls into $E(u)$ is at most
$$
\frac{\binom{d_G(u)-d_G(u,W)}{d_0}}{\binom{d_G(u)}{d_0}}\le e^{-\frac{d_0\cdot d_G(u,W)}{d_G(u)}}\le e^{-\frac{d_0}{10\ln n}{\cdot d_G(u,W)}}
$$
by {\bf (P1)}. Hence the probability that none of the vertices $u$ from $U$ chooses an edge between $u$ and $W$ to be put into its set $E(u)$ is at most:
$$
\prod_{u\in U}e^{-\frac{d_0}{10\ln n}\cdot d_G(u,W)}= e^{-\frac{d_0}{10\ln n}\cdot e_G(U,W)}=e^{-\Theta(n)}\,.
$$
Applying the union bound over all choices of $U,W$ gives the desired claim.

We now claim that for every graph $G$, and every its subgraph $\Gamma_0$ of minimum degree 2 satisfying properties {\bf (P2)}, {\bf (P3)}, {\bf (P4)}, {\bf (P5)}, {\bf (P7)}, the subgraph $\Gamma_0$ is an $(n/4,2)$-expander. In order to verify this claim, let $S\subset [n]$ be a subset of size $|S|\le n/4$. Denote $S_1=S\cap SMALL(G)$, $S_2=S-SMALL(G)$. Consider first the case where $|S_2|\le \frac{n}{\ln^{1/2}n}$. Since $\delta(\Gamma_0)\ge 2$, and all vertices from $SMALL$ are at distance more than 4 from each other by {\bf (P3)}, we obtain: $|N_{\Gamma_0}(S_1)|\ge 2|S_1|$. As for vertices from $S_2$, they are all of degree at least $d_0$ in $\Gamma_0$. The set $S_2$ spans at most $|S_2|\cdot \ln^{3/4}n$ edges in $G$, and thus in $\Gamma_0$, according to {\bf (P4)}. It thus follows that $e_{\Gamma_0}(S_2,V-S_2)\ge d_0|S_2|-2e_{\Gamma_0}(S_2)> \frac{d_0|S_2|}{2}$. Hence $|N_{\Gamma_0}(S_2)|\ge |S_2|\cdot \ln^{1/4}n$, by  {\bf (P5)}. Finally, notice that, due to the non-existence of short paths connecting $SMALL(G)$ again, the set $S_1\cup N_{\Gamma_0}(S_1)$ contains only one vertex from $u\cup N_{\Gamma_0}(u)$ for every $u\in S_2$ (here we use the fact the forbidden paths have length at most 4). Therefore, $|(S_1\cup N_{\Gamma_0}(S_1)) \cap (S_2\cup N_{\Gamma_0}(S_2))|\le |S_2|$. Altogether,
\begin{eqnarray*}
|N_{\Gamma_0}(S)|&=&|N_{\Gamma_0}(S_2)\setminus S_1|+|N_{\Gamma_0}(S_1)-(S_2\cup N_{\Gamma_0}(S_2))|\\
&\ge& |S_2| (\ln^{1/4}n-1)+2|S_1|-|S_2|\ge 2(|S_1|+|S_2|)=2|S|\,,
\end{eqnarray*}
as required. The complementary case $\frac{n}{\ln^{1/2}n}\le |S_2|\le \frac{n}{4}$ is very simple: by property {\bf (P7)}, such $S_2$ misses at most $n/\ln^{1/2}n$ vertices in its neighborhood in $\Gamma_0$, also $|S_1|\le |SMALL(G)|\le n^{0.3}$ by {\bf (P2)}. In follows that
$|N_{\Gamma_0}(S)|\ge n-\frac{n}{\ln^{1/2}n}-|S_2|-|SMALL(G)|\ge \frac{n}{2}$.
\end{proof}

\medskip

Notice that every $(\frac{n}{4},2)$-expander $\Gamma$ on $n$ vertices is necessarily connected. Indeed, if such $\Gamma$ is not connected, then consider its connected component $C$ of size $|C|\le \frac{n}{2}$, and take $U$ to be an arbitrary subset of $C$ of size $|U|=\min\{\left\lfloor\frac{n}{4}\right\rfloor,|C|\}$. Then the external neighborhood of $U$ in $\Gamma$ has size at least $2|U|>|C-U|$ by our expansion assumption, and falls entirely within $C$ --- a contradiction.

As we have stated already in this text, expanders are not necessarily Hamiltonian themselves, but they are amenable to reaching Hamiltonicity by adding extra (random) edges, as they contain many boosters. However, in our circumstances we do not have extra time for sprinkling, and the required boosters should come from within the already existing edges of the random graph. Fortunately, a random graph $G(n,m)$ with $m=m(n)$ in the relevant range has {\bf whp} a booster with respect to any sparse expander it contains, as given by the following lemma.

\begin{lemma}\label{boosters}
Let $\tilde{G}=(G_i)_{i=0}^N$ be a random graph process on $n$ vertices. Denote $G=G_{\tau_2}$, where $\tau_2=\tau_2(\tilde{G})$ is the hitting time for having minimum degree two in $\tilde{G}$. Assume the constant $\delta_0$ is small enough. Then {\bf whp} for every $(n/4,2)$-expander $\Gamma\subset G$ with $V(\Gamma)=V(G)$ and $|E(\Gamma)|\le d_0n+n$, $\Gamma$ is Hamiltonian, or  $G$ contains at least one booster with respect to $\Gamma$.
\end{lemma}

\begin{proof}
Recall that every connected $(k,2)$-expander $\Gamma$ is Hamiltonian or has at least $k^2/2$ boosters, by Corollary \ref{cor-Posa3}. In order for a random graph $G$ to violate the assertion of the lemma, $G$ should contain some $(n/4,2)$-expander $\Gamma$ with few edges, but none of at least as many as $n^2/32$ boosters relative to $\Gamma$ (note that the required connectivity of $G$ is delivered by the expansion of $\Gamma$, as explained above). Since we cannot pinpoint the exact location of $\tau_2$, we instead take the union bound over all $m_1\le m\le m_2$, as {\bf whp} $\tau_2$ is located in this interval. So the estimate is:
\begin{equation}\label{eq1}
\sum_{m=m_1}^{m_2} \sum_{i\le d_0n+n}\frac{\binom{N}{i}\cdot\binom{N-i-\frac{n^2}{32}}{m-i}}{\binom{N}{m}}+o(1)
\end{equation}
(we sum over all relevant values of $m$, adding $o(1)$ in the end to account for the probability that $\tau_2$ falls outside the interval $[m_1,m_2]$; then we sum over all possible values $i$ of $|E(\Gamma)|$; then we bound from above  by $\binom{N}{i}$ the number of $(n/4,2)$-expanders with $i$ edges in the complete graph on $n$ vertices, and finally we require the edges of $\Gamma$ to be present in $G(n,m)$, but all at least $n^2/32$ boosters relative to $\Gamma$ to be omitted). The ratio of the binomial coefficients above can be estimated as follows:
$$
\frac{\binom{N-i-\frac{n^2}{32}}{m-i}}{\binom{N}{m}}\le \frac{e^{-\frac{\frac{n^2}{32}(m-i)}{N-i}}\binom{N-i}{m-i}}
{\binom{N}{m}}\le e^{-\frac{m}{17}}\left(\frac{m}{N}\right)^i\,,
$$
assuming that $\delta_0$ in the definition of $d_0$ is small enough. We can thus estimate the $i$-th summand in (\ref{eq1}) as follows:
$$
\binom{N}{i}e^{-\frac{m}{17}}\left(\frac{m}{N}\right)^i\le \left(\frac{eN}{i}\cdot\frac{m}{N}\right)^i\cdot e^{-\frac{m}{17}}=\left(\frac{em}{i}\right)^i\cdot e^{-\frac{m}{17}}
\le\left(\frac{em}{d_0n+n}\right)^{d_0n+n}\cdot e^{-\frac{m}{17}}=o(n^{-3})\,,
$$
again for $\delta_0$ small enough (it is even exponentially, and not just polynomially, small in $n$). Summing over all $i\le d_0n+n$ and then over all $m_1\le m\le m_2$ establishes the required claim.
\end{proof}

\medskip

The stage is now set to deliver the final  punch of the proof of Theorem \ref{th-Ham-hit}. Recall that our goal is to prove that for a random graph process $\tilde{G}$, {\bf whp} the graph $G=G_{\tau_2}$ at the very moment $\tau_2$ when the minimum degree becomes 2 is already Hamiltonian. First, observe that by Lemma \ref{sparse-exp} {\bf whp} $G$ contains an $(n/4,2)$-expander $\Gamma_0$ with at most $d_0n$ edges. We start with this sparse expander $\Gamma_0$ and keep adding boosters to it until the current graph $\Gamma_i$ becomes Hamiltonian; obviously at most $n$ steps (edge additions) will be needed to reach Hamiltonicity. If we ever get stuck before reaching Hamiltonicity, say at step $i\ge 0$, then the current graph $\Gamma_i$ is still an $(n/4,2)$-expander, is connected and non-Hamiltonian, has at most $d_0n+n$ edges, but the graph $G$ has no boosters with respect to $\Gamma_i$. This however does not happen typically due to Lemma \ref{boosters}. If so, the process of edge addition eventually completes with a subgraph $\Gamma_i\subset G$, which is Hamiltonian. The proof is complete!


\begin{thebibliography}{99}
\bibitem{AKS81} M. Ajtai, J. Koml\'os and E. Szemer\'edi,
{\em The longest path in a random graph}, Combinatorica 1 (1981),
1--12.
\bibitem{AKS85}
M. Ajtai, J. Koml\'os and E. Szemer\'edi, First occurrence of Hamilton cycles in random graphs,
Cycles in graphs (Burnaby, B.C., 1982), North-Holland Mathematical Studies {115}, North-Holland,
Amsterdam (1985), 173--178.
\bibitem{BKS} I. Ben-Eliezer, M. Krivelevich and B. Sudakov,  {\em The size Ramsey number of a directed path},
 Journal of Combinatorial Theory Series B 102 (2012), 743--755.
\bibitem{B84}
B. Bollob\'as, {\em The evolution of sparse graphs}, Graph Theory and Combinatorics, Academic Press,
London (1984), 35--57.
\bibitem{Bol-book}
B. Bollob\'as, {\bf Random graphs}, 2nd ed.,
Cambridge University Press, Cambridge, 2001.
\bibitem{ER60} P. Erd\H os and A. R\'enyi, {\em On the evolution of
random graphs}, Publ. Math. Inst. Hungar. Acad. Sci. 5 (1960),
17--61.
\bibitem{Vega79} W. Fernandez de la Vega, {\em Long paths in random graphs},
Studia Sci. Math. Hungar. 14 (1979), 335--340.
\bibitem{FK-book} A. Frieze and M. Karo\'nski, {\bf Introduction to random graphs}, Cambridge University Press, to appear.
\bibitem{JLR}
S. Janson, T.\L uczak and A. Ruci\'nski,
 {\bf Random Graphs}, Wiley, New York, 2000.
\bibitem{Kar90} R. Karp, {\em The transitive closure of a random
digraph}, Random Structures and Algorithms 1 (1990), 73--93.
\bibitem{KS83}
J. Koml\'os and E. Szemer\'edi, {\em Limit distributions for the
existence of Hamilton circuits in a random graph}, Discrete
Mathematics {43} (1983), 55--63.
\bibitem{KS13} M. Krivelevich and B. Sudakov,
{\em The phase transition in random graphs -- a simple proof},
Random Structures and Algorithms 43 (2013), 131--138.
\bibitem{Pos76}
L. P\'osa, {\em Hamiltonian circuits in random graphs}, Discrete
Mathematics 14 (1976), 359--364.
\end{thebibliography}
\end{document}